\documentclass{scrartcl}

\usepackage[T1]{fontenc}
\usepackage[utf8]{inputenc}
\usepackage{amsmath, amsthm, amssymb}
\usepackage{dsfont}
\usepackage{graphicx}
\usepackage[format=plain]{caption}%ben?tigt, um mit dem Befehl \caption*{} eine Bildunterschrift zu setzen, die nicht mit Abbildung X eingeleitet wird. 'normal'=\setcapindent{0}
\usepackage{capt-of}%ben?tigt, um 'figure' und 'table' nebeneinander zu setzen, siehe 'http://www.faqs.org/faqs/de-tex-faq/part6/' Punkt 6.1.7, aufgerufen am 01.08.11.
\usepackage{subcaption}
\usepackage{bm}
\usepackage{tikz}
\usetikzlibrary{arrows,shapes,trees,calc,through}

\newcommand{\norm}[1]{\left\| #1 \right\|}  %Norm
\newcommand{\scprd}[1]{\left\langle #1 \right\rangle}  %Skalarprodukt
\renewcommand{\d}{\,\mathrm{d}} %Differenzial in Integralen
\newcommand{\e}{\mathrm{e}} %eulersche Zahl
\newcommand{\N}{\mathbb{N}}  %Zahlbereiche

\newcommand{\R}{\mathbb{R}}

\newcommand{\eps}{\varepsilon}
\renewcommand{\phi}{\varphi}
\newcommand{\ul}{\underline}
\newcommand{\ol}{\overline}

\numberwithin{equation}{section}
\newtheorem{thm}{Theorem}[section]

\newtheorem{prop}[thm]{Proposition}
\newtheorem{lm}[thm]{Lemma}

\theoremstyle{definition} \newtheorem{ex}[thm]{Example}
\theoremstyle{definition} \newtheorem{df}[thm]{Definition}
\theoremstyle{definition}

\title{Weak input-to-state stability: characterizations and counterexamples}

\author{Jochen Schmid \\  
\small Institut f\"ur Mathematik, Universit\"at W\"urzburg, 97074 W\"urzburg, Germany\\
\small jochen.schmid@mathematik.uni-wuerzburg.de}    

\date{}

\begin{document}

\maketitle

\begin{abstract}
\small{ \noindent 
We establish characterizations of weak input-to-state stability for abstract dynamical systems with inputs, %a large class of abstract control systems, 
which are similar to characterizations of uniform and of strong input-to-state stability established in a recent paper by A. Mironchenko and F. Wirth. We also answer, %resolve, 
by means of suitable counterexamples, two %some/various 
open questions concerning weak input-to-state stability (and its relation to other common stability concepts) raised in the aforementioned paper. %raised in a recent paper by A. Mironchenko and F. Wirth. In particular, we show that weak input-to-state stability is strictly weaker than strong input-to-state stability.
}
\end{abstract}

{ \small \noindent 
%\emph{AMS Subject Classification (2010):} 
%\\
}

\section{Introduction}

In this paper, we study the property of weak input-to-state stability of general dynamical systems $\mathfrak{S} = (X,\mathcal{U},\phi)$ with inputs. Such a system is determined by its generally infinite-dimensional state space $X$, its set $\mathcal{U}$ of admissible input functions, and its dynamical map $$\phi: [0,\infty) \times X \times \mathcal{U} \to X$$ which for given initial state $x_0 \in X$ and input $u \in \mathcal{U}$ yields the state $\phi(t,x_0,u)$ of the system at any time $t \in [0,\infty)$. 
Weak input-to-state stability of such a system means, roughly speaking, that $0$ is an asymptotically stable -- that is, stable and attractive -- equilibrium point of the (undisturbed) system with input $u = 0$ and that this asymptotic stability property is affected only slightly by small (disturbance) inputs $u \ne 0$. In precise terms, %In mathematically precise terms,  wISS means that 
this means that there are continuous monotonically increasing functions $\ul{\sigma}, \ul{\gamma}, \ol{\gamma}: [0,\infty) \to [0,\infty)$ with $\ul{\sigma}(0), \ul{\gamma}(0), \ol{\gamma}(0) = 0$ such that for all $(x_0,u) \in X \times \mathcal{U}$ the following estimates hold true: %one has
\begin{gather}
\norm{ \phi(t,x_0,u) } \le \ul{\sigma}(\norm{x_0}) + \ul{\gamma}(\norm{u}_{\mathcal{U}}) 
\qquad (t \in [0,\infty)) \label{eq:UGS, intro}\\
\text{and} \notag \\
\limsup_{t\to\infty} \norm{ \phi(t,x_0,u) } \le \ol{\gamma}(\norm{u}_{\mathcal{U}}) 
\label{eq:wAG, intro}
\end{gather}
meaning that %implying that/so that
the stability property and the attractivity property of $0$, respectively, are affected only slightly by disturbance inputs $u \in \mathcal{U}$ of small magnitude $\norm{u}_{\mathcal{U}}$. Inequality~\eqref{eq:UGS, intro} is commonly referred to as the uniform global stability and~\eqref{eq:wAG, intro} is referred to as the weak asymptotic gain property of the system. 
It should be noted that if the limit relation~\eqref{eq:wAG, intro} holds uniformly w.r.t.~$u \in \mathcal{U}$ or, respectively, locally uniformly w.r.t.~$x_0 \in X$ and uniformly w.r.t.~$u \in \mathcal{U}$, then the system is even strongly or uniformly input-to-state stable, respectively. 
\smallskip

In recent years, these last two notions of strong and especially of uniform input-to-state stability have been intensively studied. See, for instance,~\cite{DaMi13}, \cite{Mi16}, \cite{MiIt16}, \cite{MiWi16a}, \cite{MiWi17}, \cite{JaNaPaSc16}, \cite{JaScZw17}, \cite{NaSc17}, \cite{MaPr11}, \cite{MiKaKr17}, \cite{TaPrTa17}, \cite{KaKr16}, \cite{KaKr17}, \cite{ZhZh17a}, \cite{ZhZh17b} and the references therein. Also, weak input-to-state stability can be established for a rather large class of semilinear systems (both in the case of inputs entering in the domain and in the case of inputs entering at the boundary of the domain on which the partial differential equation describing the system lives). See~\cite{ScZw18-MTNS}, \cite{ScZw18}, \cite{Sc18-wp}. It is therefore natural to study the property of weak input-to-state stability %as well as 
-- and especially its relation to other common stability properties -- %also
from a general point of view. 
\smallskip

In the present paper, we establish a characterization of weak input-to-state stability similar to the %recent
characterizations of strong and uniform input-to-state stability from~\cite{MiWi16a} and, moreover, we investigate the relation of weak input-to-state stability to other common and natural stability concepts, namely strong input-to-state stability and zero-input uniform global stability. In particular, we answer %settle 
two open questions from~\cite{MiWi16a}. %concerning wISS raised in~\cite{MiWi16a} 
\smallskip

In more detail, the contents of the present paper can be described as follows. 
Section~\ref{sect:2} provides the necessary preliminaries %for later sections
%by recalling the concept o
setting out and recalling the precise definitions of abstract dynamical systems with inputs and of the various stability notions employed later on.  
In Section~\ref{sect:3} we establish a characterization of weak input-to-state stability which is parallel to the characterizations of strong and uniform input-to-state stability for infinite-dimensional systems recently established in~\cite{MiWi16a}. 
In Section~\ref{sect:4} we investigate the relation of weak input-to-state stability to strong input-to-state stability. 
We show, by means of a suitable counterexample, that weak input-to-state stability is strictly weaker than strong input-to-state stability, %and we thereby answer an open question 
thereby answering an open question raised in~\cite{MiWi16a}.  In our example, we work with modulated-linear systems, which are  described by  evolution equations of the form
\begin{align} \label{eq:modul-lin, def}
x' = \alpha(u(t)) A x 
\end{align}
with a linear operator $A$ and a modulating prefactor $\alpha(u(t))$, and the input space $\mathcal{U}$ is a certain subset of $L^p([0,\infty),\R)$. 
We also show that in the special case of linear systems, %for linear systems 
weak input-to-state stability is equivalent to strong input-to-state stability. 
In the case of semilinear systems, the relation of weak and strong input-to-state stability remains open. %At least, however, we show 
We show at least, however, that for semilinear systems weak input-to-state stability is strictly weaker than uniform input-to-state stability. 
In Section~\ref{sect:5} we investigate the relation of weak input-to-state stability to the combination of zero-input uniform global stability and the weak asymptotic gain property. 
We show, by means of a suitable counterexample, that weak input-to-state stability is strictly stronger than the aforementioned combination of properties, thereby answering an open question raised in~\cite{MiWi16a}. In our example, we work with linear systems with input space $\mathcal{U}$ being a certain subset of $L^{\infty}([0,\infty),\R)$. We also show that for linear systems with input space $\mathcal{U}$ being a full $L^p$-space, weak input-to-state stability is equivalent to the aforementioned combination of properties. 
\smallskip

In the entire paper, $\R^+_0 := [0,\infty)$ denotes the non-negative reals and $\ol{B}_r^{Z}(0) := \{ z \in Z: \norm{z} \le r\}$ for any subset $Z$ of a normed linear space with norm $\norm{\cdot}$. As usual, $\mathcal{K}$ and  $\mathcal{L}$ denote the following classes of comparison functions:
\begin{gather*}
\mathcal{K} := \{ \gamma \in C(\R^+_0,\R^+_0): \gamma \text{ strictly increasing with } \gamma(0) = 0 \} \\
\mathcal{L} := \big\{ \gamma \in C(\R^+_0,\R^+_0): \gamma \text{ strictly decreasing with } \lim_{r\to\infty} \gamma(r) = 0 \big\}.
\end{gather*}
Also, $\norm{\cdot}_p$ for any $p \in [1,\infty) \cup \{\infty\}$  stands for the standard norm on $L^p(\R^+_0,U)$, where $U$ is any Banach space, and $u_1\,\&_{\tau}u_2$ stands for the concatenation of the functions $u_1, u_2: \R^+_0 \to U$ at time $\tau \in \R^+_0$ defined by
\begin{align*}
(u_1\,\&_{\tau} u_2)(t) := 
\begin{cases} u_1(t) \qquad (t \in [0,\tau)) \\
u_2(t-\tau) \qquad (t \in [\tau,\infty))
\end{cases}.
%(u_1\,\&_{\tau} u_2)(t) := u_1(t) \qquad (t \in [0,\tau))
%\qquad \text{and} \qquad
%(u_1\,\&_{\tau} u_2)(t) := u_2(t-\tau) \qquad (t \in [\tau,\infty)).
\end{align*} 
And finally, in the context of admissible control operators -- and, in particular, of extrapolation of semigroup generators -- we adopt the standard notation from~\cite{EnNa}, \cite{TuWe}.

\section{Setting and definitions} \label{sect:2}

\subsection{Systems with inputs} \label{sect:contr syst}
%{Abstract dynamical systems with inputs} 

We begin by setting out the class of systems that we -- just like~\cite{MiWi16a} -- are going to deal with in this paper. 

\begin{df}
A (forward-complete) dynamical system $\mathfrak{S} = (X,\mathcal{U},\phi)$ with inputs is determined by 
\begin{itemize}
\item a normed linear space $X$ (the state space of $\mathfrak{S}$) endowed with a norm $\norm{\cdot}_X$ %$\norm{\cdot} = \norm{\cdot}_X$
\item a non-empty set $\mathcal{U} \subset %U^{\R^+_0} = 
\{ \text{functions } u: \R^+_0 \to U \}$ (the set of admissible inputs of $\mathfrak{S}$) endowed with a norm $\norm{\cdot}_{\mathcal{U}}$
\item a map $\phi: \R^+_0 \times X \times \mathcal{U} \to X$ (the dynamical map of $\mathfrak{S}$)
\end{itemize}
such that the following properties are satisfied:
\begin{itemize}
\item[(i)] $\mathcal{U}$ is invariant under shifts to the left, that is, $u(\cdot+\tau) \in \mathcal{U}$ and $\norm{u(\cdot+\tau)}_{\mathcal{U}} \le \norm{u}_{\mathcal{U}}$ for every $u \in \mathcal{U}$ and $\tau \in \R^+_0$ 
\item[(ii)] $\mathcal{U}$ is invariant under concatenations, that is, $u_1 \, \&_{\tau} \, u_2 \in \mathcal{U}$ for every $u_1,u_2 \in \mathcal{U}$ and $\tau \in \R^+_0$ 
\item[(iii)] $\phi(0,x_0,u) = x_0$ for every $(x_0,u) \in X \times \mathcal{U}$ and, moreover, $\phi$ is cocyclic, that is, 
\begin{align}
\phi(t+s,x_0,u) = \phi(t, \phi(s,x_0,u), u(\cdot+s))
\end{align}
for every $(x_0,u) \in X \times \mathcal{U}$ and $s,t \in \R^+_0$
\item[(iv)] $\phi(\cdot,x_0,u): \R^+_0 \to X$ is continuous for every $(x_0,u) \in X \times \mathcal{U}$
\item[(v)] $\phi$ is causal, that is, 
\begin{align}
\phi(\cdot,x_0,u_1)|_{[0,\tau]} = \phi(\cdot,x_0,u_2)|_{[0,\tau]}
\end{align}
for every $x_0 \in X$, $u_1, u_2 \in \mathcal{U}$ and $\tau \in \R^+_0$ with $u_1|_{[0,\tau]} = u_2|_{[0,\tau]}$.
\end{itemize}
\end{df}

In the following, we will always write $\norm{\cdot} := \norm{\cdot}_X$ for brevity. Since $\mathcal{U}$ is not assumed to be a linear space, it is sligthly abusive to speak of a norm $\norm{\cdot}_{\mathcal{U}}$ on $\mathcal{U}$. What we mean is, of course, that $\norm{\cdot}_{\mathcal{U}}$ %: \mathcal{U} \to \R^+_0$ 
is the restriction of a norm of some linear space $\mathcal{F} \supset \mathcal{U}$.  

\subsection{Stability and attractivity concepts} %properties}

We continue by recalling the stability and attractivity concepts from~\cite{MiWi16a} that will be used in the sequel.

\begin{df}
A dynamical system $\mathfrak{S} = (X,\mathcal{U},\phi)$ with inputs is called
\begin{itemize}
\item[(i)] \emph{uniformly globally stable} iff there exist $\ul{\sigma}, \ul{\gamma} \in \mathcal{K}$ such that for all $(x_0,u) \in X \times \mathcal{U}$
\begin{align} \label{eq:UGS-def}
\norm{\phi(t,x_0,u)} \le \ul{\sigma}(\norm{x_0}) + \ul{\gamma}(\norm{u}_{\mathcal{U}}) \qquad (t \ge 0)
\end{align}
\item[(ii)] \emph{uniformly locally stable} iff there exist $\ul{\sigma}, \ul{\gamma} \in \mathcal{K}$ and $r > 0$ such that~\eqref{eq:UGS-def} holds true for all $(x_0,u) \in \ol{B}_r^{X}(0) \times \ol{B}_r^{\mathcal{U}}(0)$
\item[(iii)] \emph{zero-input uniformly globally stable} or \emph{zero-input uniformly locally stable}, respectively, iff $0 \in \mathcal{U}$ and the restricted system $\mathfrak{S}_0 := (X,\mathcal{U}_0,\phi)$ with $\mathcal{U}_0 := \{0\}$ is uniformly globally or uniformly locally stable, respectively.
\end{itemize}
\end{df}

\begin{df}
Suppose $\mathfrak{S} = (X,\mathcal{U},\phi)$ is a dynamical system with inputs and $\ol{\gamma} \in \mathcal{K} \cup \{0\}$. $\mathfrak{S}$ is said to be 
\begin{itemize}
\item[(i)] of \emph{weak asymptotic gain} $\ol{\gamma}$ iff for every $\eps > 0$ and $(x_0,u) \in X \times \mathcal{U}$ there exists a time $\ol{\tau}(\eps,x_0,u) \in \R^+_0$ such that
\begin{align}
\norm{\phi(t,x_0,u)} \le \eps + \ol{\gamma}(\norm{u}_{\mathcal{U}}) \qquad (t \ge \ol{\tau}(\eps,x_0,u))
\end{align}
\item[(ii)] of \emph{strong asymptotic gain} $\ol{\gamma}$ iff for every $\eps > 0$ and $x_0 \in X$ there exists a time $\ol{\tau}(\eps,x_0) \in \R^+_0$ such that
\begin{align}
\norm{\phi(t,x_0,u)} \le \eps + \ol{\gamma}(\norm{u}_{\mathcal{U}}) \qquad (t \ge \ol{\tau}(\eps,x_0) \text{ and } u \in \mathcal{U})
\end{align}
\item[(iii)] of \emph{uniform asymptotic gain} $\ol{\gamma}$ iff for every $\eps > 0$ and $r > 0$ there exists a time $\ol{\tau}(\eps,r) \in \R^+_0$ such that
\begin{align}
\norm{\phi(t,x_0,u)} \le \eps + \ol{\gamma}(\norm{u}_{\mathcal{U}}) \qquad (t \ge \ol{\tau}(\eps,r) \text{ and } (x_0,u) \in \ol{B}_r^{X}(0) \times \mathcal{U}).
\end{align}
\end{itemize}
Also, $\mathfrak{S}$ is said to be of \emph{weak asymptotic gain} iff it is of weak asymptotic gain $\ol{\gamma}$ for some $\ol{\gamma} \in \mathcal{K} \cup \{0\}$.
\end{df}

\begin{df}
Suppose $\mathfrak{S} = (X,\mathcal{U},\phi)$ is a dynamical system with inputs. $\mathfrak{S}$ is said to have the \emph{weak limit property} iff there is a $\ol{\gamma} \in \mathcal{K}$ such that for every $\eps > 0$ and $(x_0,u) \in X \times \mathcal{U}$ there exists a time $\ol{\tau}(\eps,x_0,u)$ such that
\begin{align}
\inf_{t \in [0, \ol{\tau}(\eps,x_0,u)] } \norm{\phi(t,x_0,u)} \le \eps + \ol{\gamma}(\norm{u}_{\mathcal{U}}).
\end{align}
\end{df}

In~\cite{MiWi16a}, the weak asymptotic gain and the weak limit properties are referred to simply as asymptotic gain and limit property, respectively. We deviate from that terminology in order to emphasize the logical relation to the strong and uniform variants and in order to emphasize the parallelism of certain issues. %to be able to express certain issues in a unified and concise manner

\begin{lm}
Suppose that $\mathfrak{S} = (X,\mathcal{U},\phi)$ is a dynamical system with inputs and that $\mathcal{U} \subset L^p(\R^+_0,U)$ with $\norm{\cdot}_{\mathcal{U}} := \norm{\cdot}_p$ for some $p \in [1,\infty)$ or that $\mathcal{U} \subset L^{\infty}_0(\R^+_0,U)$ with $\norm{\cdot}_{\mathcal{U}} := \norm{\cdot}_{\infty}$, where $U$ is a Banach space and 
\begin{align*}
L^{\infty}_0(\R^+_0,U) := \big\{ u \in L^{\infty}(\R^+_0,U): \norm{u(\cdot+t)}_{\infty} \longrightarrow 0 \text{ as } t \to \infty \big\}.
\end{align*}
If $\mathfrak{S}$ is of weak asymptotic gain, then it is automatically of weak asymptotic gain $0$. 
\end{lm}

\begin{proof}
Suppose $\mathfrak{S}$ is of weak asymptotic gain $\ol{\gamma} \in \mathcal{K}$ (with corresponding times $\ol{\tau}(\eps,x_0,u)$) and let $\eps >0$ and $(x_0,u) \in X \times \mathcal{U}$ be fixed. Since $\norm{u(\cdot+t)}_{\mathcal{U}} \longrightarrow 0$ as $t \to \infty$ by our assumptions on $\mathcal{U}$, we can choose a time $t_0 \in \R^+_0$ such that
\begin{align} \label{eq:lm-wAG-implies-wAG-0, 1}
\ol{\gamma}\big( \norm{u(\cdot+t_0)}_{\mathcal{U}} \big) \le \eps. 
\end{align}
Since $\mathfrak{S}$ is of weak asymptotic gain $\ol{\gamma}$, we have for all $s \ge \ol{\tau}(\eps, \phi(t_0,x_0,u), u(\cdot+t_0))$ that
\begin{align} \label{eq:lm-wAG-implies-wAG-0, 2}
\norm{\phi(s+t_0,x_0,u)} = \norm{ \phi(s, \phi(t_0,x_0,u), u(\cdot+t_0)) } \le \eps + \ol{\gamma}\big( \norm{u(\cdot+t_0)}_{\mathcal{U}} \big).
\end{align}
Combining now~\eqref{eq:lm-wAG-implies-wAG-0, 1} and~\eqref{eq:lm-wAG-implies-wAG-0, 2} we see that for all $$t \ge \tau_0(\eps,x_0,u) := t_0 + \ol{\tau}(\eps, \phi(t_0,x_0,u), u(\cdot+t_0))$$ one has
$\norm{\phi(t,x_0,u)} \le 2 \eps$. Consequently, $\mathfrak{S}$ is of weak asymptotic gain $0$, as desired.
\end{proof}

\subsection{Input-to-state stability concepts} %notions}

With the stability and attractivity properties recalled above, we can now define the central concepts of this paper, namely weak, strong, and uniform input-to-state stability. 

\begin{df}
A dynamical system $\mathfrak{S} = (X,\mathcal{U},\phi)$ with inputs is called %\emph{weakly, or strongly, or uniformly input-to-state stable}, 
\emph{weakly input-to-state stable}, or \emph{strongly input-to-state stable}, or \emph{uniformly input-to-state stable}, 
respectively, iff it is uniformly globally stable and of weak, or strong, or uniform asymptotic gain, respectively. 
\end{df}

Instead of uniform input-to-state stability one often simply speaks of input-to-state stability in the literature.

\section{Characterization of weak input-to-state stability} \label{sect:3}

We begin with a characterization of weak input-to-state stability which is parallel to the recently established characterizations of strong and uniform input-to-state stability from~\cite{MiWi16a}. It should be pointed out that the equivalence of items~(i) and (ii) below is already stated in~\cite{MiWi16a} (Remark~5), yet without proof.

\begin{thm}
Suppose $\mathfrak{S} = (X,\mathcal{U},\phi)$ is a dynamical system with inputs. Then each of the following items is equivalent to $\mathfrak{S}$ being weakly input-to-state stable.
\begin{itemize}
\item[(i)] $\mathfrak{S}$ is uniformly globally stable and has the weak limit property
\item[(ii)] $\mathfrak{S}$ is uniformly globally stable and has the weak asymptotic gain property
\item[(iii)] there exist $\sigma, \gamma \in \mathcal{K}$ and $\beta: X \times \mathcal{U}\times \R^+_0 \to \R^+_0$ with $\beta(x_0,u,\cdot) \in \mathcal{L}$ for $x_0 \ne 0$ such that for all $(x_0,u) \in X \times \mathcal{U}$ one has:
\begin{gather}
\norm{\phi(t,x_0,u)} \le \beta(x_0,u,t) + \gamma(\norm{u}_{\mathcal{U}}) \qquad (t \in \R^+_0) \label{eq:wISS-beta}\\
\text{and} \notag \\
\beta(x_0,u,t) \le \sigma(\norm{x_0})  \qquad (t \in \R^+_0). \label{eq:wISS-bound on beta} 
\end{gather}
\end{itemize}
\end{thm}

\begin{proof}
%Implication (i) to (ii)
We first show the implication from (i) to (ii). So assume that (i) is satisfied and let $\ul{\sigma}, \ul{\gamma} \in \mathcal{K}$ and $\ol{\gamma} \in \mathcal{K}$, $\ol{\tau}(\eps,x_0,u)$ be chosen as in the definitions of uniform global stability and of the weak limit  property, respectively. 
We define the function $\gamma$ by
\begin{align}
\gamma(r) := \ul{\sigma}(2 \ol{\gamma}(r)) + \ul{\gamma}(r) \qquad (r \in \R^+_0)
\end{align}
which obviously belongs to $\mathcal{K}$. Choose and fix now $\eps > 0$ and $(x_0,u) \in X \times \mathcal{U}$ and set
\begin{align}
\tau(\eps,x_0,u) := \ol{\tau}(\delta(\eps),x_0,u) 
\qquad \text{with} \qquad 
\delta(\eps) := \frac{1}{2} \ul{\sigma}^{-1}(\eps).
\end{align}
Let $t \ge \tau(\eps,x_0,u)$. It then follows by the assumed weak limit property that there exists a $t_0 \in [0,\tau(\eps,x_0,u)]$ such that
\begin{align} \label{eq:wISS, impl (i) to (ii), 1}
\norm{\phi(t_0,x_0,u)} \le \delta(\eps) + \ol{\gamma}(\norm{u}_{\mathcal{U}})
\end{align}
It further follows by the cocycle property of $\phi$ and the assumed uniform global stability that 
\begin{align}  \label{eq:wISS, impl (i) to (ii), 2}
\norm{\phi(t,x_0,u)} &= \norm{ \phi(t-t_0, \phi(t_0,x_0,u), u(\cdot+t_0)) } \notag \\
&\le \ul{\sigma}( \norm{\phi(t_0,x_0,u)} ) + \ul{\gamma}(\norm{ u(\cdot+t_0) }).
\end{align}
Combining now~\eqref{eq:wISS, impl (i) to (ii), 1} and~\eqref{eq:wISS, impl (i) to (ii), 2} we see for every $t \ge \tau(\eps,x_0,u)$ that 
\begin{align} \label{eq:wISS, impl (i) to (ii), 3}
\norm{\phi(t,x_0,u)}  \le \ul{\sigma}(\delta(\eps) + \ol{\gamma}(\norm{u}_{\mathcal{U}}) ) + \ul{\gamma}(\norm{u}_{\mathcal{U}})
\le \eps + \gamma(\norm{u}_{\mathcal{U}}),
\end{align}
as desired. In the first inequality of~\eqref{eq:wISS, impl (i) to (ii), 3} we used that $\norm{u(\cdot+t_0)} \le \norm{u}_{\mathcal{U}}$ and in the second inequality we used the elementary fact that
\begin{align*}
\ul{\sigma}(a+b) \le \ul{\sigma}(2\max\{a,b\}) \le \ul{\sigma}(2a) + \ul{\sigma}(2b) 
\end{align*}
for all $a,b \in \R^+_0$. 
\smallskip

%Implication (ii) to (iii)
We now show the implication from (ii) to (iii). So assume that (ii) is satisfied and let $\ul{\sigma}, \ul{\gamma} \in \mathcal{K}$ and $\ol{\gamma} \in \mathcal{K}$ be chosen as in the definitions of uniform global stability and of the weak asymptotic gain property, respectively. 
We define the functions $\sigma, \gamma$ by
\begin{align} \label{eq:sigma, gamma: def}
\sigma(r) := 2 \ul{\sigma}(r) 
\qquad \text{and} \qquad 
\gamma(r) := \max \{\ul{\gamma}(r),\ol{\gamma}(r) \}
\end{align}
which obviously belong to $\mathcal{K}$. 
What we have to do now is to define for each given $(x_0,u) \in X \times \mathcal{U}$ a function $\beta(x_0,u,\cdot): \R^+_0 \to \R^+_0$ in such a way that $\beta(x_0,u,\cdot) \in \mathcal{L}$ for $x_0 \ne 0$ and that~\eqref{eq:wISS-beta} and~\eqref{eq:wISS-bound on beta} are satisfied. 
So let $(x_0,u) \in X \times \mathcal{U}$ be fixed for the rest of the proof (of the implication from (ii) to (iii)) and assume without loss of generality that
\begin{align} \label{eq:ass x_0 ne 0}
x_0 \ne 0.
\end{align}
(If $x_0 = 0$, then $\ul{\sigma}(\norm{x_0}) = 0$ and thus by the assumed uniform global stability the desired estimates~\eqref{eq:wISS-beta} and~\eqref{eq:wISS-bound on beta} hold true with the choice $\beta(x_0,u,\cdot) := 0$.)
In order to construct $\beta(x_0,u,\cdot)$ we distinguish two cases, namely whether or not $\norm{\phi(\cdot,x_0,u)}$ eventually lies below $\gamma(\norm{u}_{\mathcal{U}})$, that is, whether or not there exists a $t_0 \in \R^+_0$ such that
\begin{align} \label{eq:case distinction}
\norm{\phi(t,x_0,u)} \le \gamma(\norm{u}_{\mathcal{U}}) \qquad (t \ge t_0).
\end{align}

%Case 1
Suppose first that we are in the case where $\norm{\phi(\cdot,x_0,u)}$ eventually lies below $\gamma(\norm{u}_{\mathcal{U}})$. 
In this case, set
\begin{align}
\beta_0(x_0,u,t) := \ul{\sigma}(\norm{x_0})  \, \chi_{[\tau_0(x_0,u),\tau_{\infty}(x_0,u))}(t) \qquad (t \in \R^+_0), 
\end{align}
where
\begin{gather*}
\tau_0(x_0,u) := 0 
\qquad \text{and} \qquad
\tau_{\infty}(x_0,u) := \max\{ \sup M_{\infty}(x_0,u), 0 \} \\
M_{\infty}(x_0,u) := \big\{ t \in \R^+_0: \norm{\phi(t,x_0,u)} > \gamma(\norm{u}_{\mathcal{U}}) \big\}.
\end{gather*}
In other words, $\tau_{\infty}(x_0,u)$ is the smallest time $t_0$ for which~\eqref{eq:case distinction} holds true. In particular, we have 
\begin{align}
0 \le \tau_{\infty}(x_0,u) < \infty. 
\end{align}
It follows that $\beta_0(x_0,u,\cdot)$ is a monotonically decreasing step function satisfying
\begin{align} \label{eq:beta_0-to-zero, case 1}
\beta_0(x_0,u,t) \longrightarrow 0 \qquad (t \to \infty)
\end{align}
as well as
\begin{align} \label{eq:wISS-beta_0,case 1}
\norm{\phi(t,x_0,u)} \le \beta_0(x_0,u,t) + \gamma(\norm{u}_{\mathcal{U}})
\qquad \text{and} \qquad 
\beta_0(x_0,u,t) \le \ul{\sigma}(\norm{x_0})
\end{align}
for all $t \in \R^+_0$. (In order to see~(\ref{eq:wISS-beta_0,case 1}.a) for $t < \tau_{\infty}(x_0,u)$ use the assumed uniform global stability, and for $t \ge \tau_{\infty}(x_0,u)$ use the definition of $\tau_{\infty}(x_0,u)$ and the continuity of $\phi(\cdot,x_0,u)$.)
In view of~\eqref{eq:beta_0-to-zero, case 1} and~\eqref{eq:wISS-beta_0,case 1} we are almost done -- except that $\beta_0(x_0,u,\cdot)$ is not strictly decreasing and not continuous. %the only thing to be fixed is that $\beta_0(x_0,u,\cdot)$ is not strictly decreasing and not continuous. 
We therefore choose $\beta(x_0,u,\cdot) \in \mathcal{L}$ such that 
\begin{align} \label{eq:beta, case 1}
\beta_0(x_0,u,t) \le \beta(x_0,u,t) \le 2 \ul{\sigma}(\norm{x_0}) \qquad (t \in \R^+_0).
\end{align}
(Simply interpolate linearly between the points $(0,2\ul{\sigma}(\norm{x_0})$ and $(\tau_{\infty}(x_0,u), \ul{\sigma}(\norm{x_0}))$ and exponentially between the points $(\tau_{\infty}(x_0,u), \ul{\sigma}(\norm{x_0}))$ and $(\infty,0)$). Combining~\eqref{eq:wISS-beta_0,case 1} and~\eqref{eq:beta, case 1} we finally obtain~\eqref{eq:wISS-beta} and~\eqref{eq:wISS-bound on beta}, which concludes the proof of the implication from (ii) to (iii) in the case where $\norm{\phi(\cdot,x_0,u)}$ eventually lies below $\gamma(\norm{u}_{\mathcal{U}})$.
\smallskip

%Case 2
Suppose now that we are in the case where $\norm{\phi(\cdot,x_0,u)}$ does not eventually lie below $\gamma(\norm{u}_{\mathcal{U}})$.
In this case, there exists a unique $k(x_0,u) \in \N$ such that 
\begin{align} \label{eq:k(x_0,u), def}
\frac{\ul{\sigma}(\norm{x_0})}{k(x_0,u)+1} + \gamma(\norm{u}_{\mathcal{U}}) < \sup_{t \in \R^+_0} \norm{\phi(t,x_0,u)} \le \frac{\ul{\sigma}(\norm{x_0})}{k(x_0,u)} + \gamma(\norm{u}_{\mathcal{U}})
\end{align}
(use the assumed uniform global stability). Set now
\begin{align}
\beta_0(x_0,u,t) := \sum_{n=0}^{\infty} \frac{\ul{\sigma}(\norm{x_0})}{k(x_0,u)+n} \, \chi_{[\tau_n(x_0,u),\tau_{n+1}(x_0,u))}(t) \qquad (t \in \R^+_0),
\end{align}
where
\begin{gather*}
\tau_0(x_0,u) := 0 
\qquad \text{and} \qquad
\tau_n(x_0,u) := \sup M_n(x_0,u) \\
M_n(x_0,u) := \Big\{ t\in \R^+_0: \norm{\phi(t,x_0,u)} > \frac{\ul{\sigma}(\norm{x_0})}{k(x_0,u)+n} + \gamma(\norm{u}_{\mathcal{U}}) \Big\}
\end{gather*}
for $n \in \N$. We then have %the following facts about the times $\tau_n(x_0,u)$
\begin{gather}
0 < \tau_n(x_0,u) < \infty \qquad (n \in \N) \label{eq:tau_n,1}\\
\norm{\phi(t,x_0,u)} \le \frac{\ul{\sigma}(\norm{x_0})}{k(x_0,u)+n} + \gamma(\norm{u}_{\mathcal{U}}) \qquad (t \ge \tau_n(x_0,u) \text{ and } n \in \N_0) \label{eq:tau_n,2}\\
\tau_n(x_0,u) < \tau_{n+1}(x_0,u) \qquad (n \in \N_0)
\qquad \text{and} \qquad
\tau_n(x_0,u) \longrightarrow \infty \qquad (n \to \infty). \label{eq:tau_n,3+4}
\end{gather}
(In order to see~\eqref{eq:tau_n,1}, notice that $M_n(x_0,u)$ is bounded by the assumed weak asymptotic gain property and that $\emptyset \ne M_n(x_0,u) \ne \{0\}$ by~\eqref{eq:k(x_0,u), def} and by the continuity of $\phi(\cdot,x_0,u)$. 
In order to see~\eqref{eq:tau_n,2} for $n = 0$, just use~\eqref{eq:k(x_0,u), def} -- and to see it for $n \in \N$ use the definition of $\tau_n(x_0,u)$ and the continuity of $\phi(\cdot,x_0,u)$. 
In order to see~(\ref{eq:tau_n,3+4}.a) for $n=0$, just recall \eqref{eq:tau_n,1} -- and to see it for $n \in \N$ notice first that $M_n(x_0,u) \subset M_{n+1}(x_0,u)$ and second that 
\begin{align} \label{eq:auxiliary, tau_n strictly incr}
\norm{\phi(t,x_0,u)}\Big|_{t=\tau_n(x_0,u)} = \frac{\ul{\sigma}(\norm{x_0})}{k(x_0,u)+n} + \gamma(\norm{u}_{\mathcal{U}})
\end{align}
by virtue of~\eqref{eq:tau_n,2} and the  continuity of $\phi(\cdot,x_0,u)$. Consequently, $\tau_n(x_0,u) \le \tau_{n+1}(x_0,u)$ and $\tau_n(x_0,u) \ne \tau_{n+1}(x_0,u)$ because otherwise~\eqref{eq:auxiliary, tau_n strictly incr} would imply that $\sigma(\norm{x_0}) = 0$. Contradiction to~\eqref{eq:ass x_0 ne 0}!
And finally to see~(\ref{eq:tau_n,3+4}.b), recall that $\norm{\phi(\cdot,x_0,u)}$ does not eventually lie below $\gamma(\norm{u}_{\mathcal{U}})$. So, for every $t_0$ there exists a $t \ge t_0$ such that $\norm{\phi(t,x_0,u)} > \gamma(\norm{u}_{\mathcal{U}})$ and therefore there also exists an $n_0 \in \N$ such that 
\begin{align*}
\norm{\phi(t,x_0,u)} > \frac{\ul{\sigma}(\norm{x_0})}{k(x_0,u)+n_0} + \gamma(\norm{u}_{\mathcal{U}}).
\end{align*}
It thus follows that $\tau_n(x_0,u) \ge \tau_{n_0}(x_0,u) \ge t \ge t_0$ for all $n \ge n_0$, which proves the claimed convergence~(\ref{eq:tau_n,3+4}.b) because $t_0$ was arbitrary.)
With the help of~\eqref{eq:tau_n,1}, \eqref{eq:tau_n,2}, \eqref{eq:tau_n,3+4} it follows that $\beta_0(x_0,u,\cdot)$ is a monotonically decreasing step function satisfying
\begin{align} \label{eq:beta_0-to-zero, case 2}
\beta_0(x_0,u,t) \longrightarrow 0 \qquad (t \to \infty)
\end{align}
as well as
\begin{align} \label{eq:wISS-beta_0,case 2}
\norm{\phi(t,x_0,u)} \le \beta_0(x_0,u,t) + \gamma(\norm{u}_{\mathcal{U}})
\qquad \text{and} \qquad 
\beta_0(x_0,u,t) \le \ul{\sigma}(\norm{x_0})
\end{align}
for all $t \in \R^+_0$. (Indeed, for every $t \in \R^+_0$ there exists by~\eqref{eq:tau_n,3+4} a unique $n \in \N_0$ such that $t \in [\tau_n(x_0,u),\tau_{n+1}(x_0,u))$ and therefore
\begin{align*}
\beta_0(x_0,u,t) = \frac{\ul{\sigma}(\norm{x_0})}{k(x_0,u)+n}.
\end{align*} 
So, \eqref{eq:beta_0-to-zero, case 2} follows by virtue of~(\ref{eq:tau_n,3+4}.b) while (\ref{eq:wISS-beta_0,case 2}.a) follows by virtue of~\eqref{eq:tau_n,2}.)
We can now choose $\beta(x_0,u,\cdot) \in \mathcal{L}$ such that 
\begin{align} \label{eq:beta, case 2}
\beta_0(x_0,u,t) \le \beta(x_0,u,t) \le 2 \ul{\sigma}(\norm{x_0}) \qquad (t \in \R^+_0).
\end{align}
(Simply interpolate linearly between the points $(0,2\ul{\sigma}(\norm{x_0})$, $(\tau_1(x_0,u), \ul{\sigma}(\norm{x_0})/k(x_0,u))$, $(\tau_2(x_0,u), \ul{\sigma}(\norm{x_0})/(k(x_0,u)+1))$, \dots ). % and so on). 
Combining~\eqref{eq:wISS-beta_0,case 2} and~\eqref{eq:beta, case 2} we finally obtain~\eqref{eq:wISS-beta} and~\eqref{eq:wISS-bound on beta}, which concludes the proof of the implication from (ii) to (iii) in the case where $\norm{\phi(\cdot,x_0,u)}$ does not eventually lie below $\gamma(\norm{u}_{\mathcal{U}})$.
\smallskip

%Implication (iii) to (i)
We finally show the implication from (iii) to (i). So assume that (iii) is satisfied and let $\sigma, \gamma \in \mathcal{K}$ and $\beta$ be as in~(iii).  
Combining~\eqref{eq:wISS-beta} and~\eqref{eq:wISS-bound on beta} we immediately see that $\mathfrak{S}$ is uniformly globally stable and it remains to show that it also has the weak limit property. Set $\ol{\gamma} := \gamma \in \mathcal{K}$ and let $\eps > 0$ and $(x_0,u) \in X \times \mathcal{U}$ be given. Choose a time $\ol{\tau}(\eps,x_0,u) \in \R^+_0$ so large that
\begin{align}
\beta(x_0,u,t) \le \eps \qquad (t \ge \ol{\tau}(\eps,x_0,u)).
\end{align}
(Such a time $\ol{\tau}(\eps,x_0,u) $ exists, for if $x_0 \ne 0$ then $\beta(x_0,u,\cdot) \in \mathcal{L}$ by assumption and if $x_0 = 0$ then $\beta(x_0,u,\cdot) = 0$ by~\eqref{eq:wISS-bound on beta}.) 
It then follows by~\eqref{eq:wISS-beta} that 
\begin{align}
\norm{\phi(t,x_0,u)} \le \eps + \ol{\gamma}(\norm{u}_{\mathcal{U}})
\end{align}
for all $t \ge \ol{\tau}(\eps,x_0,u)$. Consequently, $\mathfrak{S}$ has the weak limit (and also the weak asymptotic gain) property, as desired.
\end{proof}

\section{Weak input-to-state stability and its relation to strong input-to-state stability} \label{sect:4}

\subsection{A counterexample} \label{sect:4.1}

With the following example, we show that weak input-to-state stability is, in general, strictly weaker than strong input-to-state stability. We use modulated-linear systems with suitable input spaces $\mathcal{U} \subsetneq L^p(\R^+_0,\R)$ to show this. Such modulated-linear systems correspond to evolution equations of the form~\eqref{eq:modul-lin, def}.  

\begin{ex}
Choose and fix a $p \in [1,\infty) \cup \{\infty\}$ and a function $\alpha: \R \to \R^+_0$ with $\alpha(0) = 0$ such that the set 
\begin{gather}
%\alpha(0) = 0 \\
\mathcal{U} := \bigg\{ u \in L^p(\R^+_0,\R): \alpha \circ u \text{ is  locally integrable but } \int_0^{\infty} \alpha(u(s)) \d s =  \infty \bigg\}
\end{gather}
is non-empty 
and endow $\mathcal{U}$ with the norm $\norm{\cdot}_{\mathcal{U}} := \norm{\cdot}_p$. 
(Simple choices for such a function are, for instance, $\alpha(r) := |r|$ in case $p \ne 1$ and $\alpha(r) := |r|^{1/2}$ in case $p=1$.) Also, let $A$ be the generator of a strongly stable semigroup on a Banach space $X$ and define
\begin{align}
\phi(t,x_0,u) := \e^{A ( \int_0^t \alpha(u(s)) \d s )} x_0 \qquad ((t,x_0,u) \in \R^+_0\times X \times \mathcal{U}).
\end{align}
%We then define
%\begin{align}
%\phi(t,x_0,u) := \e^{A ( \int_0^t \alpha(u(s)) \d s )} x_0 \qquad ((t,x_0,u) \in \R^+_0\times X \times \mathcal{U})
%\end{align}
%and endow $\mathcal{U}$ with the norm $\norm{\cdot}_{\mathcal{U}} := \norm{\cdot}_p$. 
We now show that $\mathfrak{S} := (X,\mathcal{U},\phi)$ is a weakly but not strongly input-to-state stable system. %dynamical system with inputs.
In particular, we see that the implications stated as open questions %problems 
in the very last paragraph of~\cite{MiWi16a} do not hold true in general. 
It is elementary to check that $\mathcal{U}$ is invariant under shifts to the left and under concatenations. It is also elementary to check that $\phi(\cdot,x_0,u)$ is continuous for every $(x_0,u) \in X \times \mathcal{U}$ and that $\phi$ is cocyclic and causal. So, in other words, $\mathfrak{S}$ is a dynamical system with inputs. 
Since $\e^{A\cdot}$ is strongly stable, it follows that
\begin{align} \label{eq:wISS ne sISS, 1}
M:= \sup_{v\in\R^+_0} \norm{\e^{A v}} < \infty
\end{align}
by the uniform boundedness principle, and therefore $\mathfrak{S}$ is uniformly globally stable with $\ul{\sigma}(r) := M r$ and $\ul{\gamma} \in \mathcal{K}$ arbitrary. 
Since $\e^{A\cdot}$ is strongly stable and since
\begin{align}
\int_0^t \alpha(u(s)) \d s \longrightarrow \infty \qquad (t \to \infty)
\end{align}
for every $u \in \mathcal{U}$, %(by the very definition of $\mathcal{U}$)
it further follows that $\mathfrak{S}$ is of weak asymptotic gain $0$. 
So, we see that $\mathfrak{S}$ is weakly input-to-state stable and it remains to show that it is not of strong asymptotic gain. Seeking a contradiction, assume that $\mathfrak{S}$ is of strong asymptotic gain $\ol{\gamma}$ with corresponding times $\ol{\tau}(\eps,x_0)$. Choose now an arbitrary $u_0 \in \mathcal{U}$ (non-empty!), let $\eps := 1$, and choose $x_0 \in X$ such that
\begin{align} \label{eq:wISS ne sISS, 2}
\norm{x_0} > \eps + \ol{\gamma}(\norm{u_0}_{\mathcal{U}}).
\end{align}
Also, define $\ol{\tau} := \ol{\tau}(\eps,x_0)$ and $u := 0 \, \&_{\ol{\tau}} \, u_0$. Clearly, $u \in \mathcal{U}$ and
\begin{align} \label{eq:wISS ne sISS, 3}
\norm{u}_{\mathcal{U}} = \norm{u_0}_{\mathcal{U}} 
\qquad \text{and} \qquad
u|_{[0,\ol{\tau}]} = 0.
\end{align}
So, by the assumed asymptotic strong gain property combined with~\eqref{eq:wISS ne sISS, 2}, \eqref{eq:wISS ne sISS, 3} and $\alpha(0) = 0$, %combining~\eqref{eq:wISS ne sISS, 2} and~\eqref{eq:wISS ne sISS, 3} with the fact that $\alpha(0) = 0$, 
we get that
\begin{align}
\eps + \ol{\gamma}(\norm{u}_{\mathcal{U}}) < \norm{x_0} = \norm{ \e^{A ( \int_0^{\ol{\tau}} \alpha(u(s)) \d s )} x_0 }
= \norm{\phi(\ol{\tau},x_0,u)} \le \eps + \ol{\gamma}(\norm{u}_{\mathcal{U}}).
\end{align}
Contradiction! $\blacktriangleleft$
\end{ex}

\subsection{Some positive results} \label{sect:4.2}

While %by the above example 
weak and strong input-to-state stability are inequivalent for modulated-linear systems with general input spaces $\mathcal{U}$, they coincide for modulated-linear systems with input space $\mathcal{U} = L^p(\R^+_0,U)$.

\begin{prop} \label{prop:pos-result 1}
Suppose $X$ is a Banach space and $\mathcal{U} := L^p(\R^+_0,\R)$ for some $p \in [1,\infty) \cup \{\infty\}$. Suppose further that $A$ is a semigroup generator on $X$ and $\alpha: \R \to \R^+_0$ is a continuous function such that $\alpha \circ u$ is locally integrable for every $u \in \mathcal{U}$ (for example $\alpha(r) := |r|$). %and define %let %set
Then $\mathfrak{S} := (X,\mathcal{U},\phi)$ with
\begin{align*}
\phi(t,x_0,u) := \e^{A ( \int_0^t \alpha(u(s)) \d s )} x_0 \qquad ((t,x_0,u) \in \R^+_0\times X \times \mathcal{U})
\end{align*} 
is a dynamical system with inputs and $\mathfrak{S}$ is weakly input-to-state stable if and only if it %$\mathfrak{S}$ 
is strongly input-to-state stable.
\end{prop}

\begin{proof}
It is easily verified that $\mathfrak{S}$ is a dynamical system with inputs. Also, one of the claimed implications is trivial. So, suppose that $\mathfrak{S}$ is weakly input-to-state stable. We have to show that $\mathfrak{S}$ is also of strong asymptotic gain and thus strongly  input-to-state stable.
\smallskip

As a first step, we show that $\alpha(0) > 0$ in case $p < \infty$ and that $\alpha(r) > 0$ for every $r \in \R$ in case $p= \infty$.
% 
%In order to do so, %In other words, 
%we show that $\alpha(u) > 0$ for every constant function $u$ belonging to $L^p = \mathcal{U}$ 
Seeking a contradiction, we assume that there is a constant function $u_r \equiv r \in L^p(\R^+_0,\R) = \mathcal{U}$ such that $\alpha(r) = 0$. We then have
\begin{align} \label{eq:modul-lin, step 1, 1}
\phi(t,x_0,u_r) = \e^{A ( \int_0^t \alpha(u_r(s)) \d s )} x_0 = x_0 \qquad (t \ge 0)
\end{align}
for every $x_0 \in X$. Since $\mathfrak{S}$ is of weak asymptotic gain $\ol{\gamma}$, say, we also have that
\begin{align}  \label{eq:modul-lin, step 1, 2}
\norm{\phi(t,x_0,u_r)} \le \eps + \ol{\gamma}(\norm{u_r}_{\mathcal{U}}) \qquad (t \ge \ol{\tau}(\eps,x_0,u_r))
\end{align}
for every $\eps >0$ and every $x_0 \in X$. Choosing now $x_0 \in X$ such that $\norm{x_0}/2 > \ol{\gamma}(\norm{u_r}_{\mathcal{U}})$ and lettig $\eps := \norm{x_0}/2$, we obtain a contradiction by combining~\eqref{eq:modul-lin, step 1, 1} and~\eqref{eq:modul-lin, step 1, 2}. 
\smallskip

As a second step, we observe that the semigroup $\e^{A\cdot}$ is strongly stable. 
Indeed, since $\mathfrak{S}$ is of weak asymptotic gain, we see by choosing $u := 0 \in \mathcal{U}$ that
\begin{align}
\e^{A(\alpha(0)t)}x_0 = \phi(t,x_0,0) \longrightarrow 0 \qquad (t \to \infty)
\end{align}
for every $x_0 \in X$. Since $\alpha(0) > 0$ by the first step, the claimed strong stability follows. 
\smallskip

As a third step, we show that for every $\ol{v}, \ol{r} \in \R^+_0$ there exists a time $\ol{\tau} = \ol{\tau}_{\ol{v},\ol{r}} \in \R^+_0$ such that
\begin{align} \label{eq:modul-lin, step 3}
\int_0^{\ol{\tau}} \alpha(u(s)) \d s \ge \ol{v} 
\qquad (u \in \mathcal{U} \text{ with } \norm{u}_{\mathcal{U}} \le \ol{r}). 
\end{align} 
So let $\ol{v}, \ol{r} \in \R^+_0$ be given and fixed. 
In case $p < \infty$, we know by the first step and the continuity of $\alpha$ that 
\begin{align*}
c_{\delta} := \inf_{|r| \le \delta} \alpha(r) > 0 
\end{align*}
for some $\delta > 0$. Setting 
$\ol{\tau} = \ol{\tau}_{\ol{v},\ol{r}} := \ol{v}/c_{\delta} + \ol{r}^p / \delta^p$
%\begin{align*}
%\ol{\tau} = \ol{\tau}_{\ol{v},\ol{r}} := \ol{v}/c_{\delta} + \ol{r}^p / \delta^p
%\end{align*}
and writing 
\begin{align*}
J^{u}_{>\delta} := \{ s\in [0,\ol{\tau}]: |u(s)| > \delta \}
\qquad \text{and} \qquad
J^{u}_{\le \delta} := \{ s\in [0,\ol{\tau}]: |u(s)| \le \delta \}
\end{align*}
for $u \in \mathcal{U}$, we see for every $u \in \mathcal{U}$ with $\norm{u}_{\mathcal{U}} \le \ol{r}$ that 
\begin{align*}
\lambda( J^{u}_{>\delta} ) \le 1/\delta^p \, \int_0^{\ol{\tau}} |u(s)|^p \d s \le \ol{r}^p / \delta^p
\end{align*}
and therefore
\begin{align*}
\int_0^{\ol{\tau}} \alpha(u(s)) \d s \ge \int_{J^{u}_{\le \delta}} \alpha(u(s)) \d s \ge c_{\delta} \lambda( J^{u}_{\le \delta} ) 
= c_{\delta} \big( \ol{\tau} - \lambda( J^{u}_{> \delta} ) \big)
\ge c_{\delta} \big( \ol{\tau} - \ol{r}^p / \delta^p \big)
= \ol{v},
\end{align*}
which proves~\eqref{eq:modul-lin, step 3} in the case $p < \infty$.
In case $p = \infty$, we know by the first step and the continuity of $\alpha$ that 
\begin{align*}
c_{\ol{r}} := \inf_{|r| \le \ol{r}} \alpha(r) > 0. 
\end{align*}
Setting $\ol{\tau} = \ol{\tau}_{\ol{v},\ol{r}} := \ol{v}/c_{\ol{r}}$, we see for every $u \in \mathcal{U}$ with $\norm{u}_{\mathcal{U}} \le \ol{r}$ that
\begin{align*}
\int_0^{\ol{\tau}} \alpha(u(s)) \d s \ge c_{\ol{r}} \ol{\tau} = \ol{v},
\end{align*}
which proves~\eqref{eq:modul-lin, step 3} in the case $p < \infty$.
\smallskip

As a fourth and last step, we finally show that $\mathfrak{S}$ is of strong asymptotic gain. 
Since $\e^{A\cdot}$ is strongly stable by the second step, we have that
\begin{align} \label{eq:modul-lin, step 4, 1}
M:= \sup_{v\in\R^+_0} \norm{\e^{A v}} < \infty
\end{align}
by the uniform boundedness principle and, moreover, we have that for every $\eps > 0$ and $x_0 \in X$ there is a $\ol{v}(\eps,x_0) \in \R^+_0$ such that
\begin{align} \label{eq:modul-lin, step 4, 2}
\norm{\e^{Av}x_0} \le \eps \qquad (v \ge \ol{v}(\eps,x_0)).
\end{align}
Set now $\ol{\gamma}(r) := M r$ for $r \in \R^+_0$, let $\eps > 0$ and $x_0 \in X$ be given and fixed, and define %. Also, define
\begin{align} \label{eq:modul-lin, step 4, 3}
\ol{\tau}(\eps,x_0) := \ol{\tau}_{\ol{v}(\eps,x_0),\norm{x_0}},
\end{align} 
where $\ol{\tau}_{\ol{v},\ol{r}}$ is chosen as in the third step. 
It then follows by~\eqref{eq:modul-lin, step 3}, \eqref{eq:modul-lin, step 4, 2} and \eqref{eq:modul-lin, step 4, 3} that for every $u \in \mathcal{U}$ with $\norm{u}_{\mathcal{U}} \le \norm{x_0}$ 
\begin{align} \label{eq:modul-lin, step 4, 4}
\norm{\phi(t,x_0,u)} = \norm{ \e^{A ( \int_0^t \alpha(u(s)) \d s )} x_0 } \le \eps \le \eps + \ol{\gamma}(\norm{u}_{\mathcal{U}})
\end{align}
for all $t \ge \ol{\tau}(\eps,x_0)$. 
It also follows by~\eqref{eq:modul-lin, step 4, 1} that for every $u \in \mathcal{U}$ with $\norm{u}_{\mathcal{U}} > \norm{x_0}$ 
\begin{align}  \label{eq:modul-lin, step 4, 5}
\norm{\phi(t,x_0,u)} = \norm{ \e^{A ( \int_0^t \alpha(u(s)) \d s )} x_0 } \le M \norm{u}_{\mathcal{U}} \le \eps + \ol{\gamma}(\norm{u}_{\mathcal{U}})
\end{align}
for all $t \ge 0$. %and in particular for all $t \ge \ol{\tau}(\eps,x_0)$. 
So, taking~\eqref{eq:modul-lin, step 4, 4} and~\eqref{eq:modul-lin, step 4, 5} together we see that $\mathfrak{S}$ is of strong asymptotic gain $\ol{\gamma}$, as desired.
\end{proof}

Similarly, weak and strong input-to-state stability coincide for linear systems with input space $\mathcal{U} = L^p(\R^+_0,U)$.

\begin{prop} \label{prop:pos-result 2}
Suppose $X$, $U$ are Banach spaces and $\mathcal{U} := L^p(\R^+_0,U)$ for some $p \in [1,\infty) \cup \{\infty\}$. Suppose further that $A$ is a semigroup generator on $X$ and that $B \in L(U,X_{-1})$ is a $\mathcal{U}$-admissible control operator for $A$, that is,
\begin{align} \label{eq:pos-result 2, B adm}
\Phi_t(u) := \int_0^t \e^{A_{-1}(t-s)} B u(s) \d s \in X 
\qquad (t \in \R^+_0 \text{ and } u \in \mathcal{U}),
\end{align}
where $A_{-1}$ is the generator of the extrapolation of the semigroup $\e^{A\cdot}$ to the extrapolation space $X_{-1}$ of $A$. 
In case $p = \infty$ additionally assume that $t \mapsto \Phi_t(u) \in X$ is continuous for every $u \in \mathcal{U}$. Then $\mathfrak{S} := (X,\mathcal{U},\phi)$ with 
\begin{align*}
\phi(t,x_0,u) := \e^{A t}x_0 + \int_0^t \e^{A_{-1}(t-s)}B u(s) \d s 
\qquad ((t,x_0,u) \in \R^+_0 \times X \times \mathcal{U})
\end{align*}
is a dynamical system with inputs and $\mathfrak{S}$ is weakly input-to-state stable if and only if it %$\mathfrak{S}$ 
is strongly input-to-state stable.
\end{prop}

\begin{proof}
Since $B$ is $\mathcal{U}$-admissible for $A$, it follows that $\phi(\cdot,x_0,u)$ is a continuous function from $\R^+_0$ to $X$ by virtue of Proposition~2.3 of~\cite{We89}. (It should be noted here that because of~\eqref{eq:pos-result 2, B adm} the linear operator $\Phi_t: \mathcal{U} \to X$ is closed and thus bounded, whence the aforementioned propostion from~\cite{We89} is applicable.) Also, it is clear that $\phi$ is cocyclic and causal. So, $\mathfrak{S}$ is a dynamical system with inputs.
It remains to show that if $\mathfrak{S}$ is weakly input-to-state stable, then it is of strong asymptotic gain and hence strongly input-to-state stable (the other implication being trivial).  
So, let $\mathfrak{S}$ be %weakly input-to-state stable, that is,
uniformly globally stable (with corresponding functions $\ul{\sigma}$, $\ul{\gamma}$) and of weak asymptotic gain $\ol{\gamma}$ (with corresponding times $\ol{\tau}(\eps,x_0,u)$). Also, let $\eps > 0$ and $x_0 \in X$. We then have for every $t \ge \ol{\tau}(\eps,x_0,0)$ and every $u \in \mathcal{U}$ that 
\begin{align*}
\norm{\phi(t,x_0,u)} 
&\le \norm{\e^{A t}x_0} + \norm{\Phi_t(u)}
= \norm{\phi(t,x_0,0)} + \norm{\phi(t,0,u)} \\
&\le \eps + \ol{\gamma}(0) + \ul{\sigma}(0) + \ul{\gamma}(\norm{u}_{\mathcal{U}})
= \eps + \ul{\gamma}(\norm{u}_{\mathcal{U}}).
\end{align*}
In particular, $\mathfrak{S}$ is of strong asymptotic gain $\ul{\gamma}$, as desired. 
\end{proof}

Clearly, every bounded control operator $B \in L(U,X)$ is $\mathcal{U}$-admissible for $A$. Sufficient conditions for unbounded control operators $B \in L(U,X_{-1})$ to be $\mathcal{U}$-admissible for $A$ can be found in~\cite{JaPa04}, \cite{TuWe}, for instance. A sufficient conditions for the continuity of $t \mapsto \Phi_t(u) \in X$ %for all $u \in \mathcal{U} = L^{\infty}(\R^+_0,U)$ 
in the case $p = \infty$ is that $X$ be a Hilbert space, $U$ be finite-dimensional and that $A$ generate an exponentially stable analytic  semigroup on $X$ that is similar to a contraction semigroup (Theorem~1 of~\cite{JaScZw17}).

\subsection{Weak input-to-state stability for semilinear systems} \label{sect:4.3}
%\subsection{Semilinear systems}

While for linear systems %by the above proposition
weak and strong input-to-state stability are equivalent, this is open for semilinear systems. %for semilinear systems this is an open question. 
With the next example, we show at least, however, that for such semilinear systems weak input-to-state stability is strictly weaker than uniform input-to-state stability.

\begin{ex}
Suppose $X$, $U$ are Hilbert spaces and $\mathcal{U} := L^2(\R^+_0,U)$. Suppose further that $A$ is a contraction semigroup generator on $X$ and $B \in L(U,X)\setminus \{0\}$ is such that $A-BB^*$ generates a strongly but not exponentially stable semigroup, %the semigroup $\e^{(A-BB^*)\cdot}$ is strongly but not exponentially stable.
and let 
\begin{align*}
f(x) := -B g(B^*x) \qquad (x \in X),
\end{align*}
where $g: U \to U$ is Lipschitz continuous on bounded subsets and 
\begin{align} \label{eq:g-strikt daempfend und noch mehr}
g(v) = v \qquad (\norm{v} \le 1) \qquad \text{and} \qquad \scprd{v,g(v)} \ge c \qquad (\norm{v} > 1)
\end{align}
for some positive constant $c > 0$. (Simple choices for such operators %$A$ and $B$ 
are given by self-adjoint operators $A$ with $\sigma(A) \subset (-\infty,0]$ and with $0$ belonging to the essential but not to the point spectrum of $A$ and by compact operators $B \in L(U,X)$ such that $BB^*$ commutes with $A$. %Beispiel hierfuer: $U = X$ und $B = P^{A}_{\{\lambda\}}$, wobei $\lambda < 0$ ein Ew von $A$ mit endlicher Vielfachheit ist
Another possible choice of operators $A$ and $B$ as above is given by Example~3.2 of~\cite{Sc18-Hagen}.)
It then follows from~\cite{Sc18-wp} (similarly to~\cite{CuZw16}) that for every $(x_0,u) \in X \times \mathcal{U}$ the initial value problem
\begin{align*}
x' = Ax + f(x) + Bu(t) \qquad \text{and} \qquad x(0) = x_0
\end{align*}
has a unique global mild solution $\phi(\cdot,x_0,u) \in C(\R^+_0,X)$ and that for $(x_0,u) \in D(A) \times C^1_c(\R^+_0,U)$ this mild solution is even a classical solution. It also follows from~\cite{Sc18-wp}  (similarly to~\cite{CuZw16}) that $\mathfrak{S} := (X,\mathcal{U},\phi)$ is a weakly input-to-state stable dynamical system with inputs. 
We now show that the system $\mathfrak{S}$ is not uniformly  input-to-state stable. Assume the contrary and set $r := \ul{\sigma}^{-1}(1/\norm{B})$. We then have for every $x_0 \in D(A)$ with $\norm{x_0} \le r$ that $\norm{B^* \phi(t,x_0,0)} \le 1$ for all $t \in \R^+_0$ by the uniform global stability of $\mathfrak{S}$ and therefore %that
\begin{align}
\phi(t,x_0,0) = \e^{(A-BB^*)t}x_0 \qquad (t \in \R^+_0)
\end{align}
by~\eqref{eq:g-strikt daempfend und noch mehr} and the classical solution property of $\phi(\cdot,x_0,0)$. 
Since now $\mathfrak{S}$ is of uniform asymptotic gain %uniformly input-to-state stable 
by our assumption, it follows that
\begin{align}
r \norm{ \e^{(A-BB^*)t} } = \sup \big\{ \norm{ \phi(t,x_0,0) }: x_0 \in D(A) \text{ with } \norm{x_0} \le r \big\}
\longrightarrow 0 %\qquad (t \to \infty)
%= \sup_{ x_0 \in D(A) \text{ with } \norm{x_0} \le r } \norm{ \phi(t,x_0,0) }
\end{align}
as $t \to \infty$. Consequently, $\e^{(A-BB^*)\cdot}$ is exponentially stable. Contradiction (to our choice of $A$ and $B$)! $\blacktriangleleft$
\end{ex}

\section{Weak input-to-state stability and its relation to zero-input uniform global stability} \label{sect:5}

%\section{Weak input-to-state stability and its relation to uniform global stability for input zero}

\subsection{A counterexample} \label{sect:5.1}

With the following example, we show that weak input-to-state stability is, in general, strictly stronger than the combination of zero-input uniform global stability and the weak asymptotic gain property. We use linear systems with a suitable input space $\mathcal{U} \subsetneq L^{\infty}(\R^+_0,\R)$ to show this.

\begin{ex}
Set $X := L^2(\R^+_0,\R)$ and %define
\begin{align}
\mathcal{U} := \big\{ u \in L^{\infty}(\R^+_0,\R): u \text{ is eventually exponentially decaying to } 0 \big\} 
\end{align}
and endow $\mathcal{U}$ with the norm $\norm{\cdot}_{\mathcal{U}} := \norm{\cdot}_{\infty}$. (What we mean by a function $u$ that eventually exponentially decays to $0$ is that there exists a time $\tau_u \in \R^+_0$ and constants $C_u, \alpha_u > 0$ such that $|u(s)| \le C_u \e^{-\alpha_u s}$ for all $s \ge \tau_u$.) Also, let $A$ be the generator of the left-translation group on $X$, that is,
\begin{align*}
\e^{A t} f = f(\cdot+t) \qquad (f \in X \text{ and } t \in \R^+_0),
\end{align*}
and let $B \in L(\R,X)$ %$B: \R \to X$ be
be given by $B v := v \, b$ for $v \in \R$, where $b \in X$ is chosen such that 
\begin{align} \label{eq:ex-2, properties of b}
b(\zeta) \ge 0 \qquad (\zeta \in \R^+_0) 
\qquad \text{and} \qquad
b \notin L^1(\R^+_0,\R)
\end{align}
(for example, $b(\zeta) = 1/\zeta \, \chi_{[1,\infty)}(\zeta)$).  
We now show that $\mathfrak{S} := (X,\mathcal{U},\phi)$ with 
\begin{align}
\phi(t,x_0,u) := \e^{A t}x_0 + \int_0^t \e^{A(t-s)}B u(s) \d s 
\qquad ((t,x_0,u) \in \R^+_0 \times X \times \mathcal{U})
\end{align}
is a system (with bounded reachability sets and with continuity at the equilibrium point $0$) which is zero-input uniformly globally stable and of weak asymptotic gain $0$ but not uniformly locally stable. In particular, $\mathfrak{S}$ is not weakly input-to-state stable and, moreover, we see that the second %horizontal 
question-marked implication in Figure~2 of~\cite{MiWi16a} does not hold true in general. 
It is elementary to check that $\mathcal{U}$ is invariant under shifts to the left and under concatenations. Also, it is clear that $\phi(\cdot,x_0,u)$ is continuous for every $(x_0,u) \in X \times \mathcal{U}$ and that $\phi$ is cocyclic and causal. So, in other words, $\mathfrak{S}$ is a dynamical system with inputs. 
Since $\e^{A \cdot}$ is a contraction semigroup, it follows that $\mathfrak{S}$ is zero-input uniformly globally stable. It also follows that $\mathfrak{S}$ has bounded reachability sets and is continuous at the equilibrium point $0$ (Definition~3 and~4 of~\cite{MiWi16a}). 
Since $\e^{A \cdot}$ is strongly stable and every $u \in \mathcal{U}$ is eventually exponentially decaying to $0$, it further follows that 
\begin{align*}
\phi(t,x_0,u) \longrightarrow 0 \qquad (t\to \infty)
\end{align*}
for every $(x_0,u) \in X \times \mathcal{U}$ or, in other words, that $\mathfrak{S}$ is of weak asymptotic gain $0$. 
So, what remains to be shown is that $\mathfrak{S}$ is not uniformly locally stable. Assume, on the contrary, that $\mathfrak{S}$ is uniformly locally stable with corresponding functions $\ul{\sigma}$, $\ul{\gamma}$ and radius~$r$. 
Since $b \in X$ satisfies~\eqref{eq:ex-2, properties of b}, we see with the help of Fatou's lemma that
\begin{align*}
\liminf_{t \to \infty} \norm{ \int_0^t b(\cdot+s) \d s }^2 
&= \liminf_{t\to\infty} \int_0^{\infty} \bigg( \int_0^t b(\zeta+s) \d s \bigg)^2 \d \zeta \notag \\
&\ge \int_0^{\infty} \bigg( \int_0^{\infty} b(\zeta+s) \d s \bigg)^2 \d \zeta
= \infty.
\end{align*}
In particular, there exists a time $\ol{\tau} \in \R^+_0$ such that
\begin{align} \label{eq:ex-2, not uls 1}
r \cdot \norm{ \int_0^{\ol{\tau}} b(\cdot+s) \d s } > \ul{\gamma}(r).
\end{align}
Choose now a $u \in \mathcal{U}$ such that
\begin{align} \label{eq:ex-2, not uls 2}
u|_{[0,\ol{\tau}]} \equiv r 
\qquad \text{and} \qquad
\norm{u}_{\mathcal{U}} = r.
\end{align}
It then follows by~\eqref{eq:ex-2, not uls 1} and \eqref{eq:ex-2, not uls 2} that  $(0,u) \in \ol{B}_r^{X}(0) \times \ol{B}_r^{\mathcal{U}}(0)$ but
\begin{align*}
\norm{\phi(\ol{\tau},0,u)} = \norm{\int_0^{\ol{\tau}} u(s) \cdot \e^{A(\ol{\tau}-s)}b \d s } = r \cdot \norm{ \int_0^{\ol{\tau}} b(\cdot+s) \d s } 
> \ul{\gamma}(r) = \ul{\sigma}(0) + \ul{\gamma}(r). 
\end{align*}
Contradiction (to our uniform local stability assumption)! $\blacktriangleleft$
\end{ex}

\subsection{A positive result} \label{sect:5.2}

While %by the above example
weak input-to-state stability and the combination of zero-input uniform global stability and the weak asymptotic gain property are inequivalent for linear systems with general input spaces $\mathcal{U}$, they coincide for linear systems with input space $\mathcal{U} = L^p(\R^+_0,U)$.

\begin{prop} \label{prop:pos-result 3}
Suppose $\mathfrak{S} := (X,\mathcal{U},\phi)$ is as in Proposition~\ref{prop:pos-result 2}. Then $\mathfrak{S}$ is a dynamical system with inputs and $\mathfrak{S}$ is weakly input-to-state stable if and only if it is zero-input uniformly locally stable and of weak asymptotic gain.  
\end{prop}

\begin{proof}
We already know that $\mathfrak{S}$ is a dynamical system with inputs, and one of the claimed implications is trivial. %and so we have only to show that if ...
So, let $\mathfrak{S}$ be zero-input uniformly locally stable (with corresponding function $\ul{\sigma}$ and radius $r$) and of weak asymptotic gain $\ol{\gamma}$ (with corresponding times $\ol{\tau}(\eps,x_0,u)$). We then have
\begin{gather*}
\norm{\e^{At}x_0} = \norm{\phi(t,x_0,0)} \le \ul{\sigma}(\norm{x_0}) 
\qquad (x_0 \in \ol{B}_r^{X}(0) \text{ and } t \ge 0) \\
\norm{\Phi_t(u)} = \norm{\phi(t,0,u)} \le 1 + \ol{\gamma}(\norm{u}_{\mathcal{U}}) 
\qquad (u \in \ol{B}_r^{\mathcal{U}}(0) \text{ and } t \ge \ol{\tau}(1,0,u)).
\end{gather*}
So, by linearity and the continuity of $[0, \ol{\tau}(1,0,u)] \ni t \mapsto \Phi_t(u)$ for $u \in \mathcal{U}$, we see that $\R^+_0 \ni t \mapsto \e^{At}x_0 \in X$ and $\R^+_0 \ni t \mapsto \Phi_t(u) \in X$ are bounded functions for all $x_0 \in X$ and $u \in \mathcal{U}$ and therefore 
\begin{align}
M_1 := \sup_{t \in \R^+_0} \norm{\e^{At}} < \infty
\qquad \text{and} \qquad 
M_2 := \sup_{t \in \R^+_0} \norm{\Phi_t} < \infty
\end{align} 
by the uniform boundedness principle. Consequently,
\begin{align*}
\norm{\phi(t,x_0,u)} \le M_1 \norm{x_0} + M_2 \norm{u}_{\mathcal{U}} \qquad (t \ge 0)
\end{align*}
for every $(x_0,u) \in X \times \mathcal{U}$. In particular, $\mathfrak{S}$ is uniformly globally stable and hence weakly input-to-state stable, as desired. 
\end{proof}

\section*{Acknowledgements}

I would like to thank the German Research Foundation (DFG) financial support through the grant ``Input-to-state stability and stabilization of distributed-parameter systems'' (DA 767/7-1).

\begin{small}

\end{small}

\end{document}